\documentclass[12pt,twoside]{article}
\usepackage[english]{babel}
\usepackage{amsmath}
\usepackage{times,amssymb,amscd}
\usepackage[utf8]{inputenc}
\usepackage[pass,legalpaper]{geometry}
\usepackage{fancyhdr}
\usepackage{multicol}
\usepackage{enumitem}
\usepackage{graphicx}
\usepackage{pgf}
\usepackage{textcomp}
\newtheorem{theorem}{Theorem}[section]

\newtheorem{lemma}[theorem]{Lemma}
\newtheorem{proposition}[theorem]{Proposition}

\newtheorem{proof}[theorem]{Proof}

\numberwithin{equation}{section}

\begin{document}
\pagestyle{myheadings}
\markboth{\centerline{Arnasli Yahya}}
{On problem of best circle}
\title
{On Problem of Best Circle to Discontinuous Groups in Hyperbolic Plane
}
\author{Arnasli Yahya \\
\normalsize Budapest University of Technology and Economics\\
\normalsize{Institute of Mathematics},
\normalsize Department of Geometry, \\
\normalsize Műegyetem rkp. 3.,H-1111 Budapest, Hungary,  \\
\normalsize arnasli@math.bme.hu,
\date{\normalsize{\today}}}
\maketitle
\begin{abstract}
The aim of this paper is to describe the largest inscribed circle into the fundamental domains of a discontinuous group in Bolyai-Lobachevsky hyperbolic plane. We give some known basic facts related to the Poincare-Delone problem and the existence notion of the inscribed circle. We study the best circle of the group $G=[3,3,3,3]$ with 4 rotational centers each of order 3. Using the Lagrange multiplier method, we would describe the characteristic of the best-inscribed circle. The method could be applied for the more general case in $G=[3,3,3,\cdots, 3]$ with $l \geq 4$ rotational centers each of order 3, by more and more computations. We observed by a more geometric Theorem 2 that the maximum radius is attained by equalizing the angles at equivalent centers and the additional vertices with trivial stabilizers, respectively. Theorem 3 will close our arguments where Lemma 3 and 4 play key roles.
\end{abstract}
\newpage
\section{Introduction}
The 17 crystal groups on Euclidean plane $\mathbb{E}^2$ have long been known (as an intuitive discovery of medieval Islamic art e.g. the artistic mosaics of Alhambra in Granada, Spain). B.N.Delone (Delaunay) described the 46 types of their fundamental domains only in 1959, see \cite{Delone}. H. Poincare in 1882 had already attempted to describe the analogous plane groups in Bolyai-Lobachevsky hyperbolic plane $\mathbb{H}^2$. A significant result of A,M. Macbeath was the description of algebraic combinatorial classification of Non-Euclidean plane crystallographic groups with compact quotient space by their signature, see \cite{Macbeath} and \cite{molnar2019}.  \\
In this paper, we would like to determine the best circle inscribed in the fundamental domain of a given discontinuous group in hyperbolic plane $\mathbb{H}^{2}$. This problem was actually raised by Prof. Emil Moln\'{a}r in \cite{molnar2019} on the base of \cite{luvcic1991}. The fundamental domain for planar discontinuous groups and uniform tilings was studied by Lu\v{c}i\'{c} and Moln\'{a}r in \cite{luvcic1991, luvcic1990combinatorial}. The algorithm for classification of fundamental polygons for a given discontinuous group was also presented by Lu\v{c}i\'{c}, Moln\'{a}r and Vasiljevi\'{c} in \cite{luvcic2018}. We are interested in the theorem, see \cite{luvcic1991}, as follows
\begin{theorem}[Lu\v{c}i\'{c}-Moln\'{a}r]\label{ExistenceInball}
Among all convex polygons in $\mathbb{E}^2$, $\mathbb{S}^2$, and $\mathbb{H}^2$ with given angles $\alpha_1, \alpha_2, \cdots, \alpha_m, m \geq 3$, there exists up to similarity (for $\mathbb{E}^2$) and up to an isometry (for $\mathbb{S}^2$ and $\mathbb{H}^2$), respecting the order of the angles, exactly one circumscribing a circle.
\end{theorem}
\noindent
That theorem will guarantee the existence of the inscribed circle in a fundamental domain for given angles. We shall determine the best circle, that is the inscribed circle with the largest radius into fundamental domains determined by a discontinuous group in $\mathbb{H}^2$.\\
In this first section we study a typical case, the hyperbolic plane group $G=[3,3,3,3]$ with 4 rotation centers of order 3 in $\mathbb{H}^2$, their fundamental domains and its representation in the tree graphs, see Fig.\ref{TreeGraph}. Basically, these tree graphs are topological images of the fundamental domain under the canonical projection mapping $\kappa:\mathbb{H}^2 \longrightarrow \mathbb{H}^2 / G$,~~  $X \longmapsto \bar{X} := X^{G}$, or simply, $\kappa$ identifies all points which form the same orbit by this group. Otherwise, to obtain the topological fundamental domain, we imagine a scissor dissecting these tree graphs, and open it up (or unfold) through the fault to construct the pre-image fundamental domain.
In section \ref{sectionLagrange} we shall consider the constrained optimum problem and apply the Lagrange multiplier method to find the solution. We will present the sufficient conditions for the local maximum points through second derivative method, called bordered Hessian criterion.
In section \ref{Other_Possibility} we use the optimality condition based on section \ref{sectionLagrange} to determine the optimum incircle radius of the other type of $G$. Moreover, we also describe the optimum condition geometrically in section \ref{Conclussion_G4}.\\
In section \ref{Generalization}, we develop the method to more general $G=[3,3,3,\cdots,3]$ of $l$ rotational points, where $l \geq 4$. All types of fundamental domains are characterized combinatorially by a Diophantine equation system. Based on these constructions, we will show the global optimum of inscribed circle radius of all fundamental domain types of $G$. We also provide an important fact on the area of the fundamental domain of all types.
Now, as a motivation, we begin with recalling the proof of the Theorem \ref{ExistenceInball} in hyperbolic plane $\mathbb{H}^2$. \\
\noindent \textbf{\textit{Proof of Theorem \ref{ExistenceInball} (for hyperbolic case)}}\\
Given $p$ is a polygon with given angles $\alpha_1, \alpha_2, \cdots, \alpha_m \in (0, \frac{\pi}{2})$, near vertices $A_1,A_2, \cdots, A_m$ which is circumscribed around a circle $k(X,x)$. Let $B_1, B_2, \cdots,$\\$ B_m$ be the set of points of tangency of $p$ and $k$, such that the angles $B_m X B_1$, $B_1XB_2$, $\cdots$, $B_{m-1}$ are equal to $\beta_1, \beta_2, \cdots, \beta_m$. Then, $\beta_1 + \beta_2 + \cdots + \beta_m=2\pi$.\\
By applying trigonometry to the rectangular central triangles $XA_i B_i$ we obtain the formula (in $\mathbb{H}^2$)
\begin{equation} \label{key1}
    \cos{\left( \frac{\alpha_i}{2}\right)}=\cosh{x} \sin{\left( \frac{\beta_i}{2} \right)}
\end{equation}
Therefore 
\begin{equation*}
    \frac{\cos{\left( \frac{\alpha_1}{2}\right)}}{\sin{\left( \frac{\beta_1}{2} \right)}}= \frac{\cos{\left( \frac{\alpha_2}{2}\right)}}{\sin{\left( \frac{\beta_2}{2} \right)}} =\cdots= \frac{\cos{\left( \frac{\alpha_m}{2}\right)}}{\sin{\left( \frac{\beta_m}{2} \right)}}=\cosh{x},
\end{equation*}
for a factor $\cosh{x}>1$ is necessary for $p$ in $\mathbb{H}^2$.\\
The existence of $x$ and also $\beta_i$, such that $\sum_{i}\beta_i=2\pi$, can be shown as follows.\\
Consider $\displaystyle{\cos{\left( \frac{\alpha_1}{2}\right)}, \cos{\left( \frac{\alpha_2}{2}\right)}, \cdots, \cos{\left( \frac{\alpha_m}{2}\right)}}$\\
From \ref{key1}, we have $\beta_i=2 \sin^{-1}\left(\frac{\cos{\left( \frac{\alpha_i}{2} \right)}}{\cosh{x}} \right)$.\\
Now, consider the following continuous function
\begin{equation}
    S(x)=\left(\sum_{i}^{m} 2 \sin^{-1}\left(\frac{\cos{\left( \frac{\alpha_i}{2} \right)}}{\cosh{x}} \right) \right)-2\pi ~~~~x \in (0, \infty)
\end{equation}
\begin{align*}
    S(0)&=\left(\sum_{i}^{m} 2 \sin^{-1}\left(\cos{\left( \frac{\alpha_i}{2} \right)} \right) \right)-2\pi =(m-2)\pi-(\alpha_1 + \alpha_2 + \cdots + \alpha_m)>0 \\ &(\text{since}~ \alpha_1 + \cdots + \alpha_m < (m-2) \pi ~\text{on}~ \mathbb{H}^2)
\end{align*}
We choose $x_0$, such that $\cosh{x_0} > \frac{1}{\sin{\left( \frac{2 \pi}{m}\right)}}$.
\begin{align*}
    S(x_0)&=\left(\sum_{i}^{m} 2 \sin^{-1}\left(\frac{\cos{\left( \frac{\alpha_i}{2} \right)}}{\cosh{x_0}} \right) \right)-2\pi \\ &<\left(\sum_{i}^{m} 2 \sin^{-1}\left(\cos{\left( \frac{\alpha_i}{2} \right)} \sin{\left( \frac{2 \pi}{m}\right)}  \right) \right)-2\pi\\
    &<\left(\sum_{i}^{m} 2 \sin^{-1}\left( \sin{\left( \frac{2 \pi}{m}\right)}  \right) \right)-2\pi = 0.
\end{align*}
We see that the function $S$ change sign in $[0,x_0]$. Since $S$ is continuous, by the intermediate value theorem, there is a value $r \in [0,x_0]$, such that $S(r)=0$.
In other words $\left(\sum_{i}^{m} 2 \sin^{-1}\left(\frac{\cos{\left( \frac{\alpha_i}{2} \right)}}{\cosh{r}} \right) \right)=2\pi$.
Hence, the inscribed circle radius is $x=r$ with the corresponding central angles $\beta_i$ satisfying $\beta_1 + \beta_2 + \cdots + \beta_m=2\pi$ $\square$

\subsection{The Hyperbolic Plane Group $G=[3,3,3,3]$}
As a typical example the group $G=[3,3,3,3]$ contains exactly 4 rotational centers each of order 3 on a topological sphere. The tree surface graphs from $G=[3,3,3,3]$ are presented in Fig.\ref{TreeGraph}. There are 5 types of graphs that represent the fundamental domains of $G$. We could construct fundamental domains based on these tree graphs. The complete corresponding fundamental domains are sketchily given in Fig.\ref{Fundamentals}.
\subsubsection{Type-5 fundamental domain}
We would like to find the best-inscribed circle into the fundamental domain of the above hyperbolic plane group G. We are first focused on the type-5 fundamental domain. Since this type has the most edges, we guess that the largest circle radius would be attained in this type.
\begin{figure}[h!]
\centering
\includegraphics[scale=0.40]{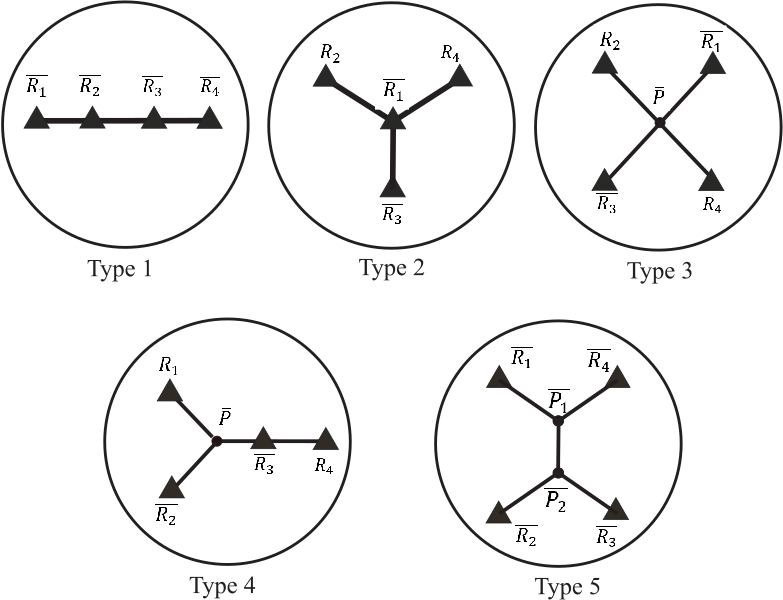}\\~\\~\\
\caption{All together: 5 types of tree surface graphs of fundamental domains for $G=[3,3,3,3]$ on a sphere}
\label{TreeGraph}
\end{figure}
\begin{figure}[h!]
\centering
\includegraphics[scale=0.45]{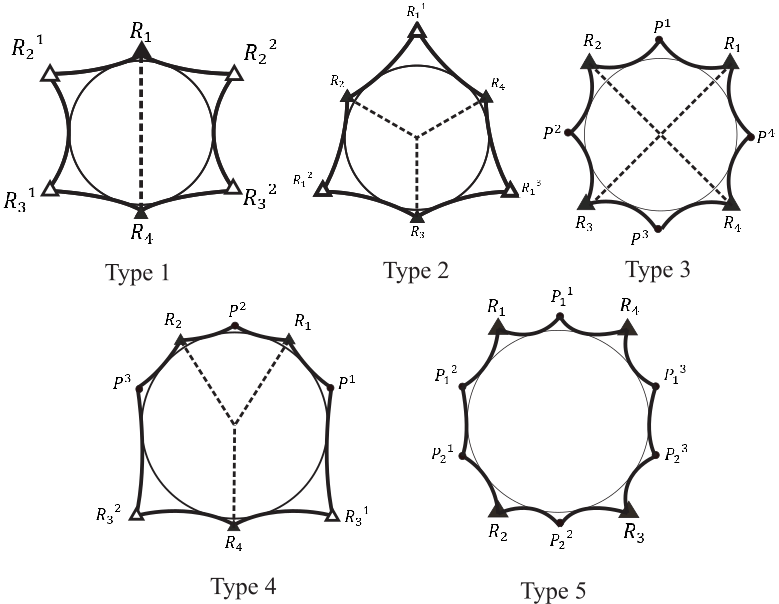}\\~\\~\\
\caption{All together: 5 types of sketchy fundamental domains for $G=[3,3,3,3]$}
\label{Fundamentals}
\end{figure}
This tree graph on the sphere is the surface diagram of the conjectured optimal fundamental domain of G=[3,3,3,3] given by its Conway-Macbeath signature. This diagram is a tree graph on a topological elastic sphere with the given 3-centers as $4=m$ vertices, each of valence (degree) 1 and $2=y$ additional vertices, each of valence 3. Then imagine a scissor we take, and cut the sphere along this tree graph to obtain a topological domain with the later metrical properties. Then the number of vertices is $6=v$, and the number of edges is $5=e$. The criterion of a tree: $v=e+1$ is fulfilled. We get a fundamental polygon of $m*1+y*3=10$ vertices (and sides), as in Fig.\ref{Treegraphtype2}.
To give more details, see Fig.\ref{Treegraphtype2}, we dissect the tree graph of type-5 through directions: $\bar{P_1} \rightarrow \bar{R_1} \rightarrow  \bar{P_1} \rightarrow \bar{P_2} \rightarrow \bar{R_2}  \rightarrow \bar{P_2} \rightarrow \bar{R_3} \rightarrow   \bar{P_2} \rightarrow  \bar{P_1} \rightarrow \bar{R_4} \rightarrow \bar{P_1}$. Then we denote the future angles $\alpha_1, \alpha_2, \alpha_6$ at vertex $\bar{P_1}$ and $\alpha_3, \alpha_4, \alpha_5$ at vertex $\bar{P_2}$, Fig. \ref{Treegraphtype2}.
We construct the fundamental domain by opening up the dissected elastic tree surface graph. As a result, we obtain type-5 fundamental domain as shown in Fig.\ref{Fundamentals}-\ref{Fundamentaldomain5}.
\begin{figure}[h!]
\centering
\includegraphics[scale=0.35]{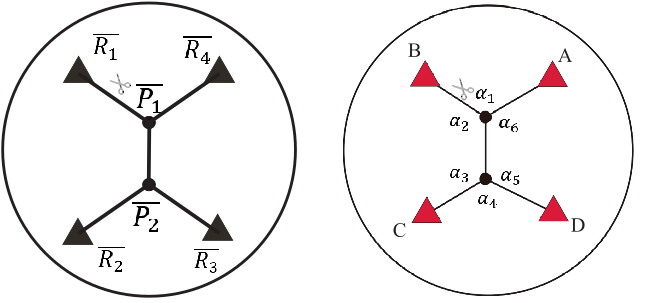}
\caption{Type-5 tree surface graph of $G=[3,3,3,3]$ is dissected by a scissor with orientation $\bar{P_1} \rightarrow \bar{R_1} \rightarrow  \bar{P_1} \rightarrow \bar{P_2} \rightarrow \bar{R_2}  \rightarrow \bar{P_2} \rightarrow \bar{R_3} \rightarrow   \bar{P_2} \rightarrow  \bar{P_1} \rightarrow \bar{R_4} \rightarrow \bar{P_1}  $ }
\label{Treegraphtype2}
\end{figure}
\begin{figure}[h!]
\centering
\includegraphics[scale=0.50]{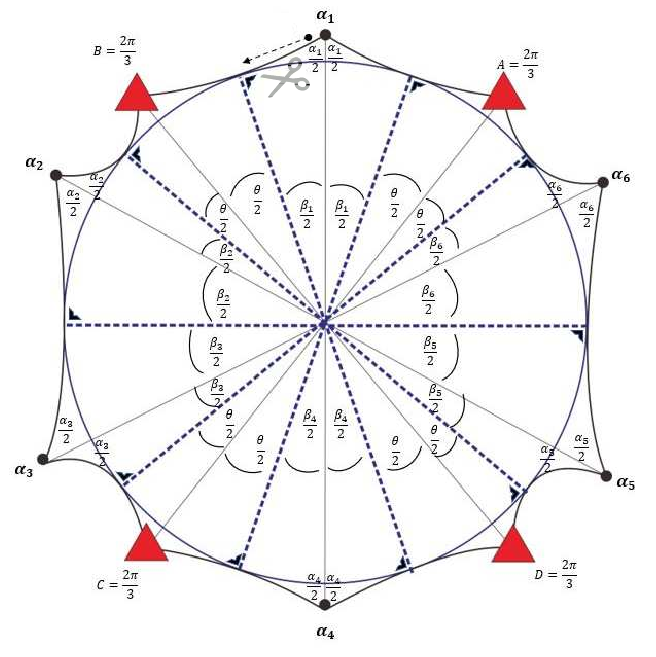}
\caption{Type-5 fundamental domain of $G=[3,3,3,3]$. Imagine also later on for $G=[3,3,3,\cdots,3]$ ($l$-times).}
\label{Fundamentaldomain5}
\end{figure}
We have some metrical properties as presented in equation system  $(\ref{cons_g1}) \cdots (\ref{cons_h2})$.
\begin{align}
    \cos{\left(\frac{{\alpha}_1}{2}\right)}=\cosh{x}  ~\sin{\left(\frac{{\beta}_1}{2}\right)} \label{cons_g1}\\ 
    \cos{\left(\frac{{\alpha}_2}{2}\right)}=\cosh{x} ~\sin{\left(\frac{{\beta}_2}{2}\right)}\label{cons_g2}\\
    \cos{\left(\frac{{\alpha}_3}{2}\right)}=\cosh{x} ~\sin{\left(\frac{{\beta}_3}{2}\right)}\label{cons_g3} \\
    \cos{\left(\frac{{\alpha}_4}{2}\right)}=\cosh{x} ~\sin{\left(\frac{{\beta}_4}{2}\right)}\label{cons_g4}\\
    \cos{\left(\frac{{\alpha}_5}{2}\right)}=\cosh{x} ~\sin{\left(\frac{{\beta}_5}{2}\right)}\label{cons_g5}\\
    \cos{\left(\frac{{\alpha}_6}{2}\right)}=\cosh{x} ~\sin{\left(\frac{{\beta}_6}{2}\right)}\label{cons_g6}\\
    \cos{\left(\frac{ \pi}{3}\right)}=\cosh{x} ~\sin{\left(\frac{\theta}{2}\right)}\label{cons_g7}
\end{align}
where 
\begin{align}
    \sum_{i=1}^{6} \beta_{i}~+~ 4 \theta =2 \pi \label{cons_h}\\
    \alpha_{1}+\alpha_{2}+\alpha_{6}=2 \pi \label{cons_h1} \\
    \alpha_{3}+\alpha_{4}+\alpha_{5}=2 \pi \label{cons_h2} 
\end{align}
From these 10 equations, we treat the equation (\ref{cons_g7}) as an equation that provides the objective function $f$. That is, we form $f=
\cosh(x)=\frac{\cos{\left( \frac{\pi}{3} \right)}}{\sin{\left( \frac{\theta}{2} \right)}}$ and we want to find the best value of radius $x$, i.e. $f$ is maximal. But, there are some conditions i.e equations $(\ref{cons_g1}),\cdots, (\ref{cons_g6}), (\ref{cons_h}), (\ref{cons_h1})$, and $(\ref{cons_h2})$ which should be satisfied. Therefore, we face to a constrained extremum problem. We shall describe the so-called Lagrange multipliers method to attack this problem in Sect.\ref{sectionLagrange}.
Now we shall motivate our approach.
We want to find the best value of radius $x$, it means that the maximum value of $x$ with the constraint above.
We shall consider the equation (\ref{cons_g7}) as the candidate of our objective function as follows
\begin{equation*}
    x=f(\theta)=\cosh^{-1}\left(\frac{1}{2 \sin(\frac{\theta}{2})}\right)
\end{equation*}
One could formally reduce the conditions above by substituting all of the constraints to $f$. Then we have $f=f(\alpha_1, \cdots, \alpha_6, \beta_1, \cdots, \beta_6, \theta)$, where the natural domain of $f$ (subset of $\mathbb{R}^{13}$) is determined by the remaining constrains.
 We first study the very specific case (regular case) where all vertices have the same interior angles $ 2 \pi / 3 $. This case provides our conjectured optimum.
\subsection{Very Specific Case (Regular Case)}
We consider a specific case, the so called regular case, by setting $\alpha_1=\alpha_2=\cdots=\alpha_6$. The constraints (\ref{cons_h1}), (\ref{cons_h2}) and (\ref{cons_h}) impose the vertex angles are equal, $\frac{2 \pi}{3}$. This choice also affects that the central angles are equal, also to $\theta$, and satisfying the following inequality (the triangle condition in the hyperbolic plane).
\begin{align*}
    \frac{\theta}{2}+\frac{1}{2}\frac{2 \pi}{3}+\frac{\pi}{2} < \pi,~\text{then}~
    \theta < \frac{\pi}{3}
\end{align*}
We just have one equation for solving the radius $x$, $\cos{\left(\frac{\pi}{3}\right)}=\cosh{(x)} \sin{\left(\frac{\theta}{2}\right)}$.
Therefore $x=\cosh^{-1}{\left(\frac{1}{2 \sin{(\frac{\theta}{2})}}\right)}$. The value $x$ is only depend on the central angle $\theta$.\\
Since the sum of all central angles should be $2 \pi$, it follows that $10 \theta = 2 \pi$, then $\theta = \frac{\pi}{5}$.
Now we can compute directly the exact value of $x$, in this specific case.
\begin{equation*}
    x \approx 1.061275061
\end{equation*}
The area A of a circle disc with radius $x$ in hyperbolic plane is given by
\begin{equation*}
    A=4 \pi \sinh^{2}\left( \frac{x}{2} \right),~\text{in our case}~A \approx 3.883222071.
\end{equation*}\\
\noindent Furthermore, the density $d$ is described by division of the area of the circle and that one of the fundamental polygon. We could compute directly that the area of polygon is $\frac{4}{3} \pi$, characteristic invariant for group $G=[3,3,3,3]$. In our calculation, we found that $d \approx 0.9270509814$.
Our conjecture is that this regular case would give the best circle, in term the largest, inscribed one into the fundamental polygon of group $G=[3,3,3,3]$. We would like to investigate this conjecture by studying some possible situations. We will have a conditional extremum problem. First, we use the tools in multivariate calculus, the so called Lagrange multiplier method.

\section{The Lagrange Multiplier Method}\label{sectionLagrange}
Based on the system of equations in the previous section, eq.$(\ref{cons_g1}), \cdots (\ref{cons_h2})$.  We formulate the following conditional extremum problem.
From these 10 equations we set the function $f$ from the constrain equation (\ref{cons_g7}) and some constraints $g_i, h, h_j$ from 9 remaining equations.
We would like to find the maximum of radius $x$.
From the equation (\ref{cons_g7}), we have $\cosh(x)=\frac{\cos{\left(\frac{ \pi}{3}\right)}}{\sin{\left(\frac{\theta}{2}\right)}}$. Since $\cosh$ is a monotonic increasing function for $x>0$, to maximize $x$, we just maximize $\cosh(x)$. We take $f(\alpha_1, \cdots \alpha_6,\beta_1, \cdots, \beta_6, \theta)=\frac{\cos{\left(\frac{ \pi}{3}\right)}}{\sin{\left(\frac{\theta}{2}\right)}}$ for the objective function.

We formulate the constraints by setting a subtraction of the expression in equations $(\ref{cons_g1}), \cdots, (\ref{cons_g6})$ from the expression in equation \ref{cons_g7}, i.e. $\frac{\cos{\left(\frac{ \pi}{3}\right)}}{\sin{\left(\frac{\theta}{2}\right)}}-\frac{\cos{\left(\frac{{\alpha}_i}{2}\right)}}{\sin{\left(\frac{{\beta}_i}{2}\right)}}=0$, for $i=1,..,6$.
Since $\alpha_i, \beta_i, \theta$ are the variables, half of them representing angles of rectangular triangles, we can restrict their value in $[0,\pi ]$.\\ Therefore we treat our problems in region $[0, \pi]^{13} \subset \mathbb{R}^{13}$. For convenience, we also write the tuple $( \alpha_1, \alpha_2, \alpha_3, \alpha_4, \alpha_5, \alpha_6, \beta_1, \beta_2, \beta_3, \beta_4, \beta_5, \beta_6, \theta)=\boldsymbol{X}$, as an element in $[0, \pi]^{13} \subset \mathbb{R}^{13}$.
The set of constrain is described by the following system.
\begin{align*}
    &g_i=\cos{\left(\frac{\pi}{3}\right)}\sin{\left(\frac{\beta_i}{2}\right)}-\cos{\left(\frac{\alpha_i}{2}\right)}\sin{\left(\frac{\theta}{2}\right)}=0,~~\text{where}~\cos{\frac{\pi}{3}}=\frac{1}{2},\\&~\text{and}~i=1,2,\cdots,6\\
    &h=\sum_{i=1}^{6} \beta_i + 4 \theta-2 \pi =0,~~
    h_1=\alpha_1+\alpha_2+\alpha_6-2 \pi=0,\\
    &h_2=\alpha_3+\alpha_4+\alpha_5-2 \pi=0
\end{align*}
The complete construction of our constrained extremum problem is described as follows
\begin{align}
    \text{Maximize}~f(\boldsymbol{X})=\frac{1}{2 \sin{\left( \frac{\theta}{2} \right)}},~ \text{subjected to the above constraints}.
\end{align}
\subsection{The Compactness of Constrained Region}
We consider the constrained region $S$. The compactness of $S$ could help us to guarantee the existence of maximum (and minimum) of $f$ in $S$. Consider $g_i(\boldsymbol{X})=\frac{1}{2} \sin{\left( \frac{\beta_i}{2}\right)}-\sin{\left( \frac{\theta}{2}\right)} \cos{\left( \frac{\alpha_i}{2}\right)}$ for $i=1,\cdots 6$ are bounded continuous functions. Therefore $g_i^{-1}(0)$, the inverse images of $0$, closed set, under continuous function are also closed in $\mathbb{R}^{13}$. One could see that $h^{-1}(0)$, $h_{1}^{-1}(0)$, $h_{2}^{-1}(0)$ are also closed in $\mathbb{R}^{13}$. With this compactness assumption, we note that $f$ is bounded in $S$, then $f$ has a maximum and minimum in $S$.
We have obtained that our conjectured point $X_0$ satisfies the necessary condition for the local maximum of the constrained extremum problem. We need to observe further whether this point is really a local maximum point. We apply the second derivative test, called the bordered determinant criterion test in \cite{Magnus2019, Trench2012}.
\section{Other Fundamental Domain Types: Finding Global Maximum}\label{Other_Possibility}
Based on our analysis on the Type-5 of the fundamental domains for $G=[3, 3, 3, 3]$, we obtain the largest radius $x\approx1.061275061$. We need to compare it with the largest radius reached on the other fundamental domains of types 1, 2, 3, 4. (Fig. \ref{TreeGraph}-\ref{Fundamentals}).
The analogous methods, Lagrange multiplier, and bordered determinant are applied to the cases of types 3 and 4, when they have independent parameters raised by the additional point. While the fundamental domains of types 1 and 2 have only fixed vertex angles. The appeared equation system could be solved immediately by some appropriate substitutions.
\begin{enumerate}
    \item \textbf{Type-1} \\
    The constructed fundamental domain on this type has no additional point. It just contains two rotational centers $R_1, R_4$ and two rotational centers $R_2, R_3$ that appear twice, see Fig.\ref{TreeGraph}, and \ref{Fundamentals}. The angles on the rotational vertices $R_1$ and $R_4$ are equals to $\frac{2\pi}{3}$. While the angles on the twice appeared vertices $R_2^1$, $R_2^2$, $R_3^1$, $R_3^2$ are  a half of its original, i.e $\frac{2\pi}{6}$.
    We derive the following system of equations 
    \begin{align*}
        \cos{\left( \frac{1}{2} \cdot \frac{2\pi}{3}\right)}=\cosh{x}\sin{\left( \frac{\theta_1}{2}\right)},&~~
        \cos{\left(\frac{1}{2} \cdot \frac{2\pi}{6}\right)}=\cosh{x}\sin{\left( \frac{\theta_2}{2}\right)}\\
        2 \theta_1 &+ 4 \theta_2 = 2 \pi
    \end{align*}
    Basically, this equation system has only fix parameters.  Using some appropriate substitutions, we conclude that the value of radius $x$ is given by 
    \begin{equation*}
        x=\cosh^{-1}{\left( \frac{3}{2} \right)} \approx 0.962423
    \end{equation*}
    \item \textbf{Type-2}\\
    In this type, we have rotational center $R_1$ that appears three times on the fundamental domain. We derive the single equation
    similar to type-1, we get the numerical value $x \approx 0.927539$.
    \item \textbf{Type-3}\\
    This type has a single additional point on the three graph Fig.\ref{TreeGraph}. The corresponding fundamental domain is given in Fig.\ref{Fundamentals}. On that fundamental domain the additional point $P$ appears four times, namely $P^1$, $P^2$, $P^3$, and $P^4$. We denote the angle near $P^1$, $P^2$, $P^3$, $P^4$ are $\alpha_1$, $\alpha_2$, $\alpha_3$, $\alpha_4$ and their corresponding central angles are $\beta_1$, $\beta_2$, $\beta_3$, $\beta_4$  respectively. The interesting is that the value of $x$ depends on $\alpha_1$, $\alpha_2$, $\alpha_3$, $\alpha_4$.  Inspired by exploration on the type-5 in section 2, we can formulate the set of constraints and find the maximum value of $\cosh{x}=\frac{1}{2 \sin{\left( \frac{\theta}{2}\right)}}$.
    Finally, this equations system could be solved for $x$, that is $x\approx 1.031718$.
    \item \textbf{Type-4}
    In this type, the vertex $R_3$ appears twice on the fundamental domain, see Fig.\ref{Fundamentals}. While the additional point $P$ is copied three times, namely $P^1$, $P^2$, $P^3$. We denote the angle near $P^1$, $P^2$, $P^3$ by $\alpha_1, \alpha_2, \alpha3$ and their corresponding central angles by $\beta_1, \beta_2, \beta_3$, respectively. We have the equation system of the constraints and the corresponding approximated value of $x$ is about $1.011595$.\\
    \end{enumerate}
The summary of all largest possible inscribed circle radii on each type of fundamental domains is presented in the following tables.
\begin{table}[h!]
\centering
    \begin{tabular}{||c c||}
    \hline
    Type & Largest Radius \\[0.5 ex]
    \hline \hline
    1 & 0.962423 \\
    \hline 
    2 & 0.927539 \\
    \hline 
    3 & 1.031718 \\
    \hline 
    4 & 1.011595 \\
    \hline
    5 & 1.061275 \\ 
    \hline
    \end{tabular}
    \caption{The largest inscribed circle radius comparison}
\label{comparison}
\end{table}
According to the largest radius comparison, Table \ref{comparison}, the largest radius of all types is attained on type-5, namely $x \approx 1.061275$. Based on the exploration of the constrained optimum problem, it could be conjectured that optimum conditions might happen whenever the corresponding parameters are equal. This intuition could be shown in the next section.
\section{Geometric Argument: Conclusion to [3,3,3,3]}\label{Conclussion_G4}

According to the approach of Lagrange multiplier method and bordered determinant criterion, it could be concluded that the maximum possible inscribed circle radius is attained whenever the corresponding independent vertex angles are equal. 
The following theorem will state this more intuitively and geometrically simpler: \textbf{Whenever we equalize the corresponding angles at $G-$equivalent vertices, the radius will increase}. This could be further developed to the proof of necessary conditions for the general problem for arbitrary cocompact plane group $G$ (as conjectured by the authors of \cite{luvcic2018}).
\begin{theorem}\label{Theorem2}
If we exchange in Theorem 1 two angles, say $\alpha_1$ and $\alpha_2$ both to $\frac{\alpha_1+\alpha_2}{2}$ in a given configuration with fixed radius $x$, so $\cosh{x}$, then for the corresponding central angles $\beta_1$ and $\beta_2$ their arithmetic mean $\frac{\beta_1 + \beta_2}{2}$ increases. So, changing $\alpha_1$, $\alpha_2$ to $\frac{1}{2}(\alpha_1 + \alpha_2)$ only, the inscribed circle will have bigger radius in the procedure.
\end{theorem}
\noindent Before proving this theorem, we first discuss a Jensen-type inequality of Lucic- Molnar by \cite{luvcic1991} for $\mathbb{H}^2$ which is described in the following lemma.
\begin{lemma} 
 The function $\beta : (0, \frac{\pi}{2}) \ni \alpha \mapsto \beta(\alpha) \in (0, \frac{\pi}{2})$, as above, given by   $\sin{(\beta(\alpha))} = \frac{\cos{\alpha}}{\cosh{x}}$ , with fixed $x$ and $\cosh{x}$ , is concave (from below).
\end{lemma}
\noindent \textbf{\textit{Proof}} [By comunication with Prof. Emil Moln\'{a}r]\\
As we look at formulas in Theorem \ref{ExistenceInball}, $\displaystyle \cos{\left(\frac{\alpha}{2}\right)}=\cosh{x} \sin{\left(\frac{\beta}{2} \right)}$ is the crucial relation (with fixed $\cosh{x}>1$) to define central angles $\beta_i(\alpha_i)$ of $\alpha_i$, $(i=1,\cdots,m \geq 3)$ as function $\beta(\alpha)$ of $\alpha$.
Let us start with
\begin{equation}
    \sin{(\beta(\alpha))}=\frac{\cos{\alpha}}{\cosh{x}},~~~ (0<\alpha<\frac{\pi}{2}) \label{awal}
\end{equation}
By differentiating both sides by $\alpha$, we obtain
\begin{align*}
   \frac{d}{d\alpha} (\sin{(\beta(\alpha))})&= \frac{d}{d\alpha}\left( \frac{\cos{\alpha}}{\cosh{x}} \right),~\text{it leads to}~
   \cos(\beta(\alpha)) \frac{d\beta(\alpha)}{d\alpha}&=-\frac{\sin{\alpha}}{\cosh{x}}
\end{align*}
\begin{align}
    \frac{d\beta(\alpha)}{d\alpha}=\frac{1}{\cosh{x}}\left( -\frac{\sin{\alpha}}{\cos{(\beta{(\alpha)})}}\right).\label{beta'}
\end{align}
We differentiate again $\frac{d\beta(\alpha)}{d\alpha}$ by $\alpha$
\begin{align*}
    \frac{d^2}{d\alpha^2}(\beta(\alpha))&=\frac{d}{d\alpha}\left(\frac{d\beta(\alpha)}{d\alpha}\right)
    =\frac{1}{\cosh{x}}\left( \frac{\sin{(\beta{(\alpha)})}(\sin{\alpha})\frac{d\beta(\alpha)}{d\alpha}-\cos{(\beta{(\alpha)})}\cos{\alpha}}{\cos^{2}(\beta(\alpha))} \right)
\end{align*}
Using  the facts $\displaystyle{\frac{1}{\cosh^{2}x}=1-\tanh^2{x}}$ and $\displaystyle{\frac{1}{\cos^{2}\beta{(\alpha)}}=1+\tan^2{\beta{(\alpha)}}}$ and also substituting Eq.\ref{awal}, \ref{beta'}, we obtain
\begin{align*}
    \frac{d^2}{d\alpha^2}(\beta(\alpha))&=-\frac{1}{\cosh^{2}x}\left( \frac{\tan{(\beta{(\alpha)})} \sin^2{\alpha}+\cos{(\beta{(\alpha)})}\cos{\alpha} \cosh{x}}{\cos^2{\beta{(\alpha)}}}\right)\\
    &=-\frac{1}{\cosh^{2}x} \left(\frac{\tan^2{(\beta{(\alpha)})}+1}{\tan{(\beta{(\alpha)})}} \right)\left( \tan^2{(\beta{(\alpha)})} \sin^2{\alpha}+\cos^2{\alpha}\right) < 0~~ \square
\end{align*}
Thus $\beta{(\alpha)}$ is concave (from below) function.\\~\\
\noindent \textbf{\textit{Proof of Theorem \ref{Theorem2}}}.\\
By Lemma 1, and because sinus is a monotone increasing function in $(0,\frac{\pi}{2})$
\begin{equation*}
    \sin{\left( \beta{\left( \frac{\alpha_1 + \alpha_2}{2} \right)} \right)}> \frac{\sin{(\beta{(\alpha_1)})} + \sin{(\beta{(\alpha_2)})}}{2} = \left( \frac{\sin{\beta_1}+\sin{\beta_2}}{2} \right)
\end{equation*}
holds as a Jensen-type inequality (The graph of the function is over the segment $(\alpha_1; \beta_1)$ $(\alpha_2; \beta_2)$  in midpoint $\left( \frac{\alpha_1 + \alpha_2}{2} \right)$, then this stands every point of the segment). Then the sum would be  $\sum_{i}\beta_i>2\pi$. \textbf{To equalize it again}, by the procedure in Theorem \ref{ExistenceInball} with previous angles and two times $\frac{1}{2}(\alpha_1+\alpha_2)$ instead of $\alpha_1$ and $\alpha_2$, $\cosh{x}$ and $x$ have to be chosen bigger. In our local optimal cases where every possible equality has been reached, such an increasing of $x$ $(\cosh{x})$ by choosing $\left( \frac{\alpha_1 + \alpha_2}{2} \right)$ is not possible, so $x$ cannot increase in such a way. Comparison of these local optima serves the optimum since the existence has already been guaranteed by compactness of the domain of variables. $\square$
\section{Generalization to $G=[3,3,3 \cdots, 3]$ \\ of $l \geq 4$ rotational centers each of order 3}\label{Generalization}
Finally, we shall see that group $G=[3,3,3, \dots,3]$, with $l$-times rotational center of order 3, $l \geq 4$. The largest inscribed circle radius could be attained by equalizing the angles corresponding to the additional vertices as many as possible. We follow the following propositions in \cite{luvcic1990combinatorial, molnar2019} to study this general construction.
\begin{proposition}\label{Proposition_Emil}
 For the number $w$ of additional points of an orbifold tree holds
 \begin{equation*}
     w \leq 2 \alpha g + l -q-2
 \end{equation*}
 If $n$ is the number of edges (and vertices) of a fundamental domain of a plane group $G$, then (with some exceptions if the domain is unique) holds $\displaystyle{n_{\mathrm{min}} \leq n \leq n_{\mathrm{max}}}$, where
 \begin{equation}
     n_{\mathrm{min}}=2 \alpha g ~ \text{if} ~l=q=0,~~ \text{or}~~ n_{\mathrm{min}}=q_0 + \left( \sum_{k=1}^{q} l_k \right)+2\alpha g + 2 l + 2 q -2~ \text{otherwise}
 \end{equation}
 and\\
 \begin{equation}
     n_{\mathrm{max}}=\left( \sum_{k=1}^{q} l_k \right) + 6 \alpha g + 4l+5q-6,
 \end{equation}
 where $\alpha=2$ if the orbifold is orientable and $\alpha=1$ otherwise, and $q_0$ is the number of boundary components containing no dihedral corner. Moreover, for a given $G$ there exist fundamental domains with $n_{\mathrm{min}}$ and $n_{\mathrm{max}}$ edges.
\end{proposition}
We study that in our cases $G=[3,3,3, \cdots, 3]$ above the $l$ rotational center are embedded into a topological sphere, i.e $g=0$. Since it is an orientable surface, $\alpha=2$. Moreover, it has no boundary component, $q=q_0=0$. Applying these conditions to the proposition, we have Lemma \ref{Proposition_Emil} as follows
\begin{lemma}\label{lemma_addpoint}
In $G=[3,3,3, \cdots, 3]$ of $l$-rotational centers of order 3, $l \geq 4 $,  there are possible number of additional points $w$ that are bounded as follows
\begin{equation}\label{additional}
   0 \leq w \leq l-2,
\end{equation}
Furthermore the possible number $n$ of sides (and vertices) of the fundamental polygon is given by
\begin{equation}
    2l-2 \leq n \leq 4l-6
\end{equation}
\end{lemma}

Finally, we give the last theorem of this paper, namely the maximum radius of inscribed circle into the fundamental domain of $G=[3,3,3,\cdots,3]$.
\begin{theorem}
Let $G=[3,3,3,\cdots,3]$ be a group with $l$-rotational centers of order 3, $l \geq 4$. The largest inscribed circle radius in its fundamental domain is realized when $l-2$ additional points are given, and their corresponding vertex angles are equalized. Furthermore, the inscribed circle radius $x$ is given by formula
\begin{equation}
    x=\cosh^{-1}{\left(\frac{1}{2 \sin{(\frac{\pi}{4l-6}})}\right)},~~~\text{for all}~~l=4, 5, 6, \cdots
\end{equation}
\end{theorem}
We need some preparations to prove this theorem. We divided our discussion into 3 following subsections also with additional information.  
\subsection{On combinatorial structure to the tree graph \\ of $G=[3,3,3,\cdots,3]$}
Firstly, the tree graph on the topological sphere for corresponding fundamental domain can be obtained completely through the algorithm in \cite{luvcic2018}, indicated previously. Particularly, in this case, $G=[3,3,3,\cdots,3]$, the tree graphs can be represented by the set of solutions for a "linear Diophantine equation system". \\
Let $\mathrm{A}_i$ be the number of rotational centers that have degree $i$ in the tree surface graph, i.e they have $i$ edges connected. Hence, the total number of all $A_i$ should be $l$, $\sum_{i=1}^{l-1} \mathrm{A}_i = l$. Note that the maximum possible degree of a rotational center is $l-1$.\\
Again, let $\mathrm{B}_j$ be the number of additional points whose degree is $j$ in the tree graph. The minimum degree of an additional point is $3$. While the maximum possible degree is $l$, e.g it happens in a star graph. Therefore, by adding all $B_j$, we get $w$, the total number of additional points, i.e $\sum_{j=3}^{l} \mathrm{B}_j = w$.
Furthermore, in our tree surface graph, the vertices can be either rotational centers or additional points. Note that the sum of all degrees of vertices in a graph is equal to 2 times the number of its edges. Since in a tree graph with $v$ vertices the number of edges is $v-1$, we can state the following equation
\begin{equation*}
    \sum_{i=1}^{l-1} i \cdot \mathrm{A}_i + \sum_{j=3}^{l} j \cdot \mathrm{B}_j = 2 (l+w-1)=n,
\end{equation*}
where $n$ is the number of vertices (sides of the fundamental polygon to the tree graph of vertices $v=l+w$ edges $v-1=l+w-1$).\\
Therefore, all of possible tree graphs for $G$ have to satisfy the solutions $\{ A_i, B_j \}$, $i=1 \cdots l-1$, $j=3 \cdots l$ of the following "linear Diophantine equation system"
\begin{align}\label{Diophantine1}
    \sum_{i=1}^{l-1} i \cdot \mathrm{A}_i + \sum_{j=3}^{l} j \cdot \mathrm{B}_j &= 2 (l+w-1)=n\\
    \sum_{i=1}^{l-1} \mathrm{A}_i &= l \\
    \sum_{j=3}^{l} \mathrm{B}_j &= w \\ \label{Diophantine4}
    \mathrm{A}_i, \mathrm{B}_j, &\in \mathbb{N} \cup \{0\},~ \text{where}(0 \leq w \leq l-2)
\end{align}
Example as before:
Let $G=[3,3,3,3]$, i.e $l=4$. The possible additional points are $w=0,1,2$. The corresponding linear Diophantine equations system is given by
\begin{align*}
    \mathrm{A}_1 + 2 \mathrm{A}_2 + 3 \mathrm{A}_3 + 3 \mathrm{B}_3 + 4 \mathrm{B}_4 = 2 ( 4 + w-1)\\
    \mathrm{A}_1 + \mathrm{A}_2 + \mathrm{A}_3 = 4,~
    \mathrm{B}_3 + \mathrm{B}_4 = w,~
    \text{where}~w=0,1,2
\end{align*}
The complete 5 solutions of the system above and their corresponding tree surface graphs, see Fig.\ref{TreeGraph},  are presented in the following table
\begin{table}[h!]
\centering
    \begin{tabular}{||c c c c c c c||}
    \hline
    Additional points & $\mathrm{A}_1$ & $\mathrm{A}_2$ & $\mathrm{A}_3$ & $\mathrm{B}_3$ &  $\mathrm{B}_4$ & Tree surface graph \\[0.5 ex]
    \hline \hline
    0 & 2 & 2 & 0 & 0 & 0 & Type-1 \\
    \hline 
    0 & 3 & 0 & 1 & 0 & 0 & Type-2 \\
    \hline 
    1 & 4 & 0 & 0 & 0 & 1 & Type-3 \\
    \hline 
    1 & 3 & 1 & 0 & 1 & 0 & Type-4 \\
    \hline
    2 & 4 & 0 & 0 & 2 & 0 & Type-5 \\ 
    \hline
    \end{tabular}
    \caption{The Diophantine equations system solution and its tree surface graph representations for $G=[3,3,3,3]$}
\label{}
\end{table}
for $l=4$ rotational centers, there are maximum $l-2$ additional points, (\ref{additional}). Consider $w=l-2$ maximum additional points added, then the corresponding solution of (\ref{Diophantine1})-(\ref{Diophantine4}), is $\mathrm{A}_1=l$, $\mathrm{A}_i=0$ for $i\neq 1$, and $\mathrm{B}_3=l-2$, $\mathrm{B}_j=0$ for $j\neq 3$. The corresponding inscribed circle radius of each linear Diophantine solution (tree surface graph types) could be described in the next two subsections.
\subsection{The constrained optimum problem in a single equation}
Consider a tree surface graph and its fundamental domain of $G$. Let $R_i$ be a rotational center with $i$ adjacent edges ($i \in \{1, 2, 3, \cdots, l-1\}$). The scissor dissecting in this tree surface graph yields the fundamental domain, particularly the rotational center with $i$ edges are dissected into $i$ identical angles, i.e $\frac{1}{i} \frac{2 \pi}{3}$. Furthermore, the corresponding trigonometric relation formed by right triangle in Fig.\ref{Trigono1} can be written as follows
\begin{align*}
\cos{\left(\frac{\alpha_i}{2}\right)}&=\cosh{x}\sin{\left(\frac{\beta_i}{2}\right)},~\text{it leads to}~
    \cos{\left(\frac{1}{i}  \frac{ \pi}{3} \right)}=\cosh{x} \cdot \sin{\left(\frac{\beta_i}{2}\right)} \end{align*}
    \text{then}~$\beta_i =2\sin^{-1}{\left( \frac{\cos{\left(\frac{1}{i}\frac{\pi}{3}\right)}}{\cosh{x}} \right)}$, for $i=1, 2,\cdots, l-1$.
Particularly, if $i=1$, i.e the rotational center appears as a "leaf" in the tree surface graph, we have
\begin{align}\label{cosh}
    \cosh{x}=\frac{1}{2 \sin{\left( \frac{\beta_1}{2} \right)}}
\end{align}
Remark: The conditions $\cosh{x}>1$ in (\ref{cosh}) affect to the boundness of $\beta_1$, i.e we can define the interval for $\beta_1$, that is $\beta_1 \in (0,\frac{\pi}{3})$.\\
\begin{figure}[h!]
    \begin{center}
        \begin{minipage}[b]{0.45\textwidth}
            \includegraphics[scale=0.35]{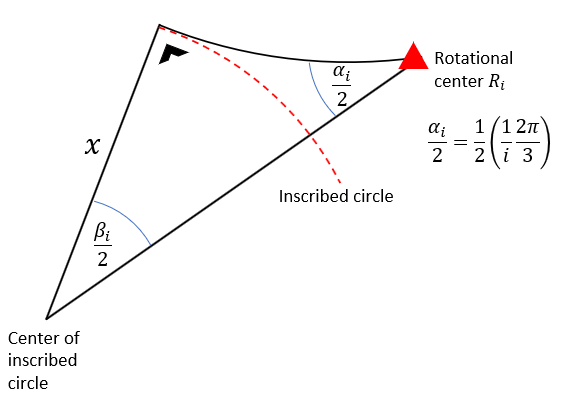}
            \caption{Right triangle with rotational center. For larger $\alpha_i$ we get smaller $\beta_i$.}
            \label{Trigono1}
        \end{minipage}
        \begin{minipage}[b]{0.45\textwidth}
            \includegraphics[scale=0.35]{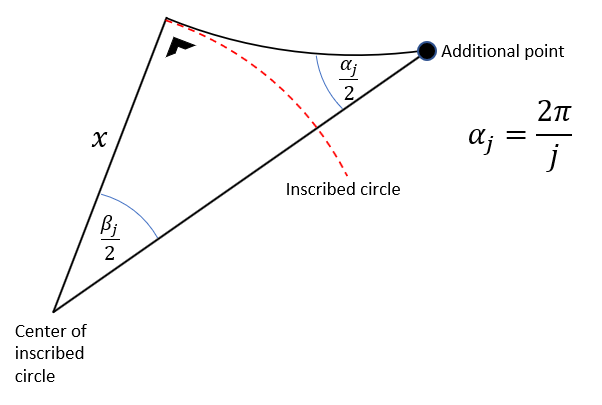}
            \caption{Right triangle with additional point}
            \label{Trigono2}
        \end{minipage}
    \end{center}
\end{figure}
By substituting the expression $\cosh{x}$ (\ref{cosh}) into $\beta_i$'s we obtain 
\begin{equation}\label{beta_i}
    \beta_i=2\sin^{-1}\left( 2 \cos{\left( \frac{1}{i} \frac{\pi}{3} \right)} \sin{\left( \frac{\beta_1}{2} \right)} \right)
\end{equation}
Remark: the argument of $\sin^{-1}$ in (\ref{beta_i}) need to be naturally  on the interval $[-1,1]$ (in this situation $[0,1]$). That means, 
\begin{align*}
   0 \leq  2\cos{\left( \frac{1}{i} \frac{\pi}{3}\right)} \sin{\left( \frac{\beta_1}{2}\right)} \leq 1, ~ \text{for every}~ i=1,\cdots,l-1
\end{align*}
it means $\beta_1$, is bounded i.e.
\begin{align}\label{beta_11}
  0 \leq  \beta_1 \leq 2 \sin^{-1}{\left( \frac{1}{2\cos{\left( \frac{1}{i} \frac{\pi}{3} \right)}} \right)},~\text{for every}~i=1,\cdots,l-1.
\end{align}
It means $\beta_1$ bounded by the least upper bound, i.e. $0 \leq \beta_1 \leq 2\sin^{-1}{\left( \frac{1}{2\cos{\left( \frac{1}{l-1} \frac{\pi}{3} \right)}} \right)}$, for fixed $l \geq 4$.
A similar argumentation is applied in the right triangle with the additional point as vertex see Fig.\ref{Trigono2}. Unlike the rotational center case, in this case, we have $\alpha_j=\frac{2\pi}{j}$
Then the trigonometric relationship in the triangle related to additional points is given by
\begin{equation}\label{beta_j}
    \beta_j=2\sin^{-1}\left( 2 \cos{\left( \frac{\pi}{j} \right)} \sin{\left( \frac{\beta_1}{2} \right)} \right),~\text{for}~j=3, 4,\cdots, l.
\end{equation}
Again, since the argument of $\sin^{-1}$ should be on $[-1,1]$ (in our case $[0,1]$), by the analogous consideration as in (\ref{beta_11}), we have
\begin{align}\label{beta_12}
  0 \leq  \beta_1 \leq 2 \sin^{-1}{\left( \frac{1}{2\cos{\left( \frac{\pi}{j} \right)}} \right)},~\text{for every}~j=3,\cdots,l.
\end{align}
It means we have $0 \leq \beta_1 \leq 2\sin^{-1}{\left( \frac{1}{2\cos{\left( \frac{\pi}{l} \right)}} \right)}$, for fixed $l \geq 4$.
The sum of all central angles of the inscribed circle i.e. $\beta_i$'s and $\beta_j$'s should be equal to $2\pi$, once complete rotation. That is, the following conditions should be fulfilled
for every $\{\mathrm{A}_i;\mathrm{B}_j \}$  solutions of (\ref{Diophantine1})-(\ref{Diophantine4})
\begin{equation}\label{central_angles}
    \sum_{i=1}^{l-1}i\mathrm{A}_i \beta_i + \sum_{j=3}^{l}j\mathrm{B}_j \beta_j = 2\pi
\end{equation}
By substituting $\beta$'s from (\ref{beta_i}) and (\ref{beta_j}) we have a nice relation as follows
\begin{align}
    \sum_{i=1}^{l-1} &i\mathrm{A}_i 2\sin^{-1}\left( 2 \cos{\left( \frac{1}{i} \frac{\pi}{3} \right)} \sin{\left( \frac{\beta_1}{2} \right)} \right) \nonumber  \\ &+ \sum_{j=3}^{l}j\mathrm{B}_j 2\sin^{-1}\left( 2 \cos{\left( \frac{\pi}{j} \right)} \sin{\left( \frac{\beta_1}{2} \right)} \right) = 2\pi.
\end{align}
Note that based on (\ref{beta_11}) and (\ref{beta_12}), $\beta_1$ is defined on 
\begin{equation*}
    0 \leq \beta_1 \leq 2 \sin^{-1}{\left( \frac{1}{2\cos{\left( \frac{1}{l-1} \frac{\pi}{3} \right)}} \right)}
\end{equation*}
In this last equation, we need to find $\beta_1 $ only to determine the corresponding inradius $x$ in each Diophantine solution. For convenience, we write $\beta_1$ as $\beta$, and the upper bound $2 \sin^{-1}{\left( \frac{1}{2\cos{\left( \frac{1}{l-1} \frac{\pi}{3} \right)}} \right)}=:K_l$, for fixed $l \geq 4$. Finally, we formulate our problem concretely as follows:
\begin{lemma}\label{Essential_Lemma}
In each tree surface graphs of $G=[3,3,3,\cdots,3]$ of $l \geq 4$ rotational centers of order 3 there is a Diophantine system (\ref{Diophantine1})-(\ref{Diophantine4}), its solution $\{ \mathrm{A}_i; \mathrm{B}_j \}$ $i=1,\cdots l-1$, $j=3, \cdots l$; and the radius of inscribed circle $x$ is obtained by
\begin{equation}
    \cosh{x}=\frac{1}{2 \sin{\left( \frac{\beta}{2} \right)}}
\end{equation}
where $\beta$ is the root of Equation
\begin{align}\label{h_function}
    \sum_{i=1}^{l-1}&i\mathrm{A}_i 2\sin^{-1}\left( 2 \cos{\left( \frac{1}{i} \frac{\pi}{3} \right)} \sin{\left( \frac{\beta}{2} \right)} \right) \nonumber \\  &+ \sum_{j=3}^{l}j\mathrm{B}_j 2\sin^{-1}\left( 2 \cos{\left( \frac{\pi}{j} \right)} \sin{\left( \frac{\beta}{2} \right)} \right) - 2\pi=0,
\end{align} in the interval $[0, K_l]$, where $K_l=2 \sin^{-1}{\left( \frac{1}{2\cos{\left( \frac{1}{l-1} \frac{\pi}{3} \right)}} \right)}$.
One could observe that the smaller root $\beta$ obtained, the larger inradius $x$ determined. $\square$
\end{lemma}
\subsection{Proof of theorem 3}
By observations in Sect.3, we have seen that the theorem holds in $G=[3,3,3,3]$, $l=4$. Hence, it is sufficient to prove the remaining cases i.e $l\geq 5$.\\
Firstly, we denote the previous function $h$ in (\ref{h_function}), in Lemma \ref{Essential_Lemma} as follows:
For every fixed solution $\{ \mathrm{A}_i, \mathrm{B}_j \}, i=1 \cdots l-1, j=3, \cdots l$ of Diophantine system (\ref{Diophantine1})-(\ref{Diophantine4}), we define a function, extended at the endpoint of its interval  
\begin{align*}
    h:\left[0, K_l \right] \longrightarrow \mathbb{R},~\text{where}~ K_l=2 \sin^{-1}{\left( \frac{1}{2\cos{\left( \frac{1}{l-1} \frac{\pi}{3} \right)}} \right)}, \text{and}~ h~ \text{is defined by}
\end{align*}
\begin{align}
    h(\beta)=\sum_{i=1}^{l-1}&i\mathrm{A}_i 2\sin^{-1}\left( 2 \cos{\left( \frac{1}{i} \frac{\pi}{3} \right)} \sin{\left( \frac{\beta}{2} \right)} \right) \nonumber \\ &+ \sum_{j=3}^{l}j\mathrm{B}_j 2\sin^{-1}\left( 2 \cos{\left( \frac{\pi}{j} \right)} \sin{\left( \frac{\beta}{2} \right)} \right) - 2\pi.
\end{align}
Observe that, $ K_l=2 \sin^{-1}{\left( \frac{1}{2\cos{\left( \frac{1}{l-1} \frac{\pi}{3} \right)}} \right)} < 2 \sin^{-1}{\left( \frac{1}{2\cos{\left( \frac{\pi}{3} \right)}} \right)}=\pi$. Hence, $[0,K_l] \subset [0,\pi]$, in particular $\frac{\beta}{2} \in [0,\frac{K_l}{2}] \subset [0, \frac{\pi}{2}]$.
Note that $h$ is a strictly increasing function in $[0, K_l]$, since $h$ appears as a linear combination of composition terms of $\sin^{-1}$ and $\sin$, where $\sin{(\frac{\beta}{2})}$ increasing on $[0, K_l] \subset [0,\frac{\pi}{2}]$, and also $\sin^{-1}$ is increasing on $[0,1]$.
As Lemma 3 stated, once we solve $h(\beta)=0$ for $\beta$, then the inradius $x$ can be simply computed. In this setting, we want to minimize the root $\beta$. Remark: The existence of the root $\beta$ in $[0, K_l]$ is guaranteed by the continuity of $h$. In fact, $h(0)=-2\pi<0$, and also,
\begin{align*}
h(K_l)&=\sum_{i=1}^{l-1}i\mathrm{A}_i 2\sin^{-1}\left( 2 \cos{\left( \frac{1}{i} \frac{\pi}{3} \right)} \sin{\left( \frac{K_l}{2} \right)} \right)\\ & + \sum_{j=3}^{l}j\mathrm{B}_j 2\sin^{-1}\left( 2 \cos{\left( \frac{\pi}{j} \right)} \sin{\left( \frac{K_l}{2} \right)} \right) - 2\pi\\
&\geq \sum_{i=1}^{l-1}i\mathrm{A}_i 2\sin^{-1}\left( {\cos{(\frac{\pi}{3})}}  \right)  + \sum_{j=3}^{l}j\mathrm{B}_j 2\sin^{-1}\left(  {\cos{(\frac{\pi}{3})}} \right) - 2\pi\\
&=\frac{\pi}{3}\left(\sum_{i=1}^{l-1}i\mathrm{A}_i+\sum_{j=3}^{l}j\mathrm{B}_j \right)-2\pi=\frac{\pi}{3} \left( 2(l-1+w)\right)-2\pi 
\geq 0
\end{align*}
then by the intermediate value theorem, we may validate this claim. Moreover, since $h$ is strictly increasing in $[0,K_l]$, the root is unique in that interval.
Our claim is that the maximum inscribed circle radius $r$ is realized by the solution to Diophantine system whose number of additional points are maximum $w=l-2$. In this situation, the Diophantine system has exactly a unique solution i.e $\mathrm{A}_1=l$, $A_i=0$ for $i=2, \cdots, l-1$ and $B_3=l-2$, $B_j=0$ for $j=4, \cdots, l$.
The corresponding function $h$ is $ h_{l-2}(\beta)=(l+3(l-2))\cdot2 \sin^{-1}(\sin{(\frac{\beta}{2})})-2\pi$ $=(4l-6)\beta-2\pi$.
 Clearly, the root of $h_{l-2}(\beta)=0$ is $\beta_{l-2}=\frac{2\pi}{4l-6}$ that gives the corresponding radius $r_{l-2}=\cosh^{-1}\left( \frac{1}{2 \sin{(\frac{\beta_{l-2}}{2})}}\right)=\cosh^{-1}\left( \frac{1}{2 \sin{(\frac{\pi}{4l-6})}}\right)$, as expressed in the theorem.\\~\\
\textbf{Suppose indirectly} that there exists a solution to Diophantine system with less additional points $w$, $0\leq w < l-2$, say $\{ \mathrm{A}_i^{*},\mathrm{B}_j^{*} \}$ and the corresponding equation $h^*(\beta)=0$,  such that it has a root $\beta^{*}$ whose resulting radius $r^{*}$ is larger than $r_{l-2}$, $r^* > r_{l-2}$. It is equivalent to $\beta^{*} < \beta_{l-2}$. Since $h^*$ is strictly increasing, $0=h^*(\beta^*)<h^*(\beta_{l-2})$, i.e. $h^*(\beta_{l-2})>0$. Meanwhile, we have $h_{l-2}(\beta_{l-2})=0$ already. It would lead to the following inequality
\begin{align*}
     h^*(\beta_{l-2}) > 0, ~~\text{or explicitely}
\end{align*}
\begin{align}\label{48'''}
    \sum_{i=1}^{l-1}&i \mathrm{A}_i^* 2\sin^{-1}\left( 2 \cos{\left( \frac{1}{i} \frac{\pi}{3} \right) } \sin{\left( \frac{\beta_{l-2}}{2} \right)}\right)\\&+\sum_{j=3}^{l}j \mathrm{B}_j^* 2\sin^{-1}\left( 2 \cos{\left(  \frac{\pi}{j} \right) } \sin{\left( \frac{\beta_{l-2}}{2} \right)}\right) \nonumber -2\pi >0,
\end{align}
    as indirect assumption.\\~\\
Substitute $\beta_{l-2}=\frac{2\pi}{4l-6}$ and apply from Appendix the Jensen-type inequalities:
\begin{equation}
    \sin^{-1}\left[ 2 \sin{\left( \frac{\pi}{4l-6} \right)} \cos{\left( \frac{\pi}{3i}\right)}\right] < \frac{3}{\pi}\sin^{-1}\left[ 2 \sin{\left(\frac{\pi}{4l-6} \right)}\right]\left( \frac{\pi}{2}-\frac{\pi}{3i} \right),
\end{equation}
and
\begin{equation}
    \sin^{-1}\left[ 2 \sin{\left( \frac{\pi}{4l-6} \right)} \cos{\left( \frac{\pi}{j}\right)}\right] < \frac{3}{\pi}\sin^{-1}\left[ 2 \sin{\left(\frac{\pi}{4l-6} \right)}\right]\left( \frac{\pi}{2}-\frac{\pi}{j} \right),
\end{equation}
for $i=1,\cdots,l-1$ and $j=3,\cdots,l$.
Therefore, in (\ref{48'''}) the sums will be much simpler, we can refer to the equations (\ref{Diophantine1}-\ref{Diophantine4}) in Diophantine system for $(\mathrm{A}_i^*, \mathrm{B}_j^*)$ and we obtain the following
\begin{align*}
    &\frac{3}{\pi} \sin^{-1} \left[ 2 \sin{\left( \frac{\pi}{4l-6}\right)} \right] \left\{ \frac{\pi}{2} \left[ 2(l+w-1) \right]-\frac{\pi}{3}l-\pi w \right\} -2\pi >0,\\
    &\text{i.e.}~ \sin^{-1}\left[ 2 \sin{\left( \frac{\pi}{4l-6} \right)} \right] \left( 2l-3\right)>2\pi,~\text{then}~2\sin{\left( \frac{\pi}{4l-6} \right)}>\sin{\left(4\cdot \frac{\pi}{4l-6} \right)}.
\end{align*}
Since $l \geq 5$, thus $4l-6 \geq 14$, we get a contradiction in interval $(0, \frac{\pi}{14}]$ by the easy analysis of sine function. $\square$ 
\subsubsection{Remark on the area of fundamental domain $\mathcal{F}_G$}
We have just found the optimum radius of the inscribed circle of $G$. This optimal radius provides the optimum density of the circle into its Fundamental domain $\mathcal{F}_G$ immediately. Since the following observation shows that the area of $\mathcal{F}_G$ is constant for every Diophantine solution. 
The area of the fundamental domain $\mathcal{F}_G$ is proportional to the angle defect $\bigtriangleup$. In fact, $\mathcal{F}_G$ can be dissected into a number of right triangles, as illustrated in Fig \ref{Fundamentaldomain5}, where the area of each right triangle could be simply computed through its defect angle, see Fig.\ref{Trigono1}, \ref{Trigono2}.
Let $\bigtriangleup_i$ be the defect angle of the right triangle about rotational center $R_i$, (Fig. \ref{Trigono1}). Similarly, let $\bigtriangleup_j$ be the defect angle of right triangle about rotational point, (Fig. \ref{Trigono2}). The dissecting of $\mathcal{F}_G$ gives the result
\begin{equation}
    \text{Area}~\mathcal{F}_G=\sum_{i=1}^{l-1} i~\mathrm{A}_i~2~\bigtriangleup_i + \sum_{j=3}^{l} j~\mathrm{B}_j~2~\bigtriangleup_j,
\end{equation}
where $\{\mathrm{A}_i ; \mathrm{B}_j \}$ is a Diophantine solution in \ref{Diophantine1}-\ref{Diophantine4}.\\
Now, by considering Fig. \ref{Trigono1}-\ref{Trigono2}, the defect angles are exactly $\bigtriangleup_i=\pi-\frac{\pi}{2}-\frac{\alpha_i}{2}-\frac{\beta_i}{2}$ and $\bigtriangleup_j=\pi-\frac{\pi}{2}-\frac{\alpha_j}{2}-\frac{\beta_j}{2}$. Also, by Diophantine conditions in (\ref{Diophantine1})-(\ref{Diophantine4}) together with central angle condition in (\ref{central_angles}), we can conclude
\begin{align*}
    \text{Area}~\mathcal{F}_G&=\sum_{i=1}^{l-1} i~\mathrm{A}_i~2~\bigtriangleup_i + \sum_{j=3}^{l} j~\mathrm{B}_j~2~\bigtriangleup_j\\ &
    =2\sum_{i=1}^{l-1} i~\mathrm{A}_i~\left( \frac{\pi}{2}-\frac{\alpha_i}{2}-\frac{\beta_i}{2} \right)+2\sum_{j=3}^{l} j~\mathrm{B}_j~\left( \frac{\pi}{2}-\frac{\alpha_j}{2}-\frac{\beta_j}{2} \right)\\
    &=\pi (2(l+w-1))-\frac{2}{3}\pi l-2 \pi w -2 \pi=\left( \frac{4}{3}l-4 \right)\pi.~~ \square
\end{align*}

\section{Appendix}
 \begin{lemma}
     The upper bounds of ~$\sin^{-1}\left(2\sin{\left( \frac{\pi}{4l-6} \right)} \cos{\left( \frac{\pi}{3i} \right)} \right)$ \\and~ $\sin^{-1}\left(2\sin{\left( \frac{\pi}{4l-6} \right)} \cos{\left( \frac{\pi}{j} \right)} \right)$ is given by
 \begin{align*}
     &\sin^{-1}\left( 2 \sin{\left( \frac{\pi}{4l-6} \right)} \cos{\left( \frac{\pi}{3i}\right)}\right) < \frac{3}{\pi}\sin^{-1}\left( 2 \sin{\left(\frac{\pi}{4l-6} \right)}\right)\left( \frac{\pi}{2}-\frac{\pi}{3i} \right),\\
     &\text{and}\\
     &\sin^{-1}\left( 2 \sin{\left( \frac{\pi}{4l-6} \right)} \cos{\left( \frac{\pi}{j}\right)}\right) < \frac{3}{\pi}\sin^{-1}\left( 2 \sin{\left(\frac{\pi}{4l-6} \right)}\right)\left( \frac{\pi}{2}-\frac{\pi}{j} \right),
 \end{align*}
 for all $i=1 \cdots l-1$, $j=3, \cdots, l$ and $l \geq 5$.
 \end{lemma}
 
 \begin{proof}
 We will provide the proof for the first inequality, then the second inequality can be proven in a similar way.
 Consider \begin{align*}
     2\sin{\left(\frac{\pi}{4l-6} \right)}\cos{\left( \frac{\pi}{3i} \right)} < 2\sin{\left(\frac{\pi}{4l-6} \right)} < \cos{\left( \frac{\pi}{3i} \right)},~ \text{for all}~ i=1,\cdots l-1,~ l \geq 5
 \end{align*}
Since $\sin^{-1}$ is increasing in $(0,1]$, then we have
\begin{equation*}
\sin^{-1}\left( 2\sin{\left(\frac{\pi}{4l-6} \right)}\cos{\left( \frac{\pi}{3i} \right)}\right) < \sin^{-1}\left(2\sin{\left(\frac{\pi}{4l-6} \right)} \right) < \sin^{-1}\left( \cos{\left( \frac{\pi}{3i} \right)} \right).
\end{equation*}
Since $\sin^{-1}$ is concave up, then the slope of its secant line through the origin $(0,0)$ and $(x, \sin^{-1}(x))$ is increasing, that is $\frac{\sin^{-1}(x_1)}{x_1} < \frac{\sin^{-1}(x_2)}{x_2},~\text{if}~x_1<x_2$.
Therefore,
\begin{equation*}
    \frac{\sin^{-1}\left( 2\sin{\left(\frac{\pi}{4l-6} \right)}\cos{\left( \frac{\pi}{3i} \right)}\right)}{2\sin{\left(\frac{\pi}{4l-6} \right)}\cos{\left( \frac{\pi}{3i} \right)}} < \frac{\sin^{-1}\left(2\sin{\left(\frac{\pi}{4l-6} \right)} \right) }{2\sin{\left(\frac{\pi}{4l-6} \right)}} < \frac{\sin^{-1}\left( \cos{\left( \frac{\pi}{3i} \right)} \right)}{\cos{\left( \frac{\pi}{3i} \right)}}.
\end{equation*}
 Now, we multiply all (positive) sides in the inequality by\\ (positive) $2\sin{\left(\frac{\pi}{4l-6} \right)}\cos{\left( \frac{\pi}{3i} \right)}$ to have
 \begin{align*}
     \sin^{-1}\left( 2\sin{\left(\frac{\pi}{4l-6} \right)}\cos{\left( \frac{\pi}{3i} \right)}\right) &<\sin^{-1}\left(2\sin{\left(\frac{\pi}{4l-6} \right)} \right) \cdot \cos{\left( \frac{\pi}{3i} \right)}\\
     &< 2\sin{\left(\frac{\pi}{4l-6} \right)} \cdot \sin^{-1}\left( \cos{\left( \frac{\pi}{3i} \right)} \right) .
 \end{align*}
 Hence, we have
 \begin{align*}
     \frac{\sin^{-1}\left( 2\sin{\left(\frac{\pi}{4l-6} \right)}\cos{\left( \frac{\pi}{3i} \right)}\right)}{2\sin{\left(\frac{\pi}{4l-6} \right)} \cdot \sin^{-1}\left( \cos{\left( \frac{\pi}{3i} \right)} \right)} < \frac{\sin^{-1}\left(2\sin{\left(\frac{\pi}{4l-6} \right)} \right) \cdot \cos{\left( \frac{\pi}{3i} \right)}}{2\sin{\left(\frac{\pi}{4l-6} \right)} \cdot \sin^{-1}\left( \cos{\left( \frac{\pi}{3i} \right)} \right)}.
 \end{align*}
 Note that $\cos{\left( \frac{\pi}{3i} \right)} \geq \cos{\left( \frac{\pi}{3} \right)}$ for all $i=1 \cdots l-1$. Since $\sin^{-1}$ is concave up, then we have $\frac{\sin^{-1}\left( \cos{\left( \frac{\pi}{3i}\right)} \right)}{\cos{\left( \frac{\pi}{3i}\right)}} \geq \frac{\sin^{-1}\left( \cos{\left( \frac{\pi}{3} \right)} \right)}{\cos{\left( \frac{\pi}{3} \right)}}=\frac{\pi}{3}$. Therefore, $\frac{\cos{\left( \frac{\pi}{3i}\right)}}{\sin^{-1}\left( \cos{\left( \frac{\pi}{3i}\right)} \right)} \leq \frac{3}{\pi}$.
 Then we have
 \begin{align*}
     &\frac{\sin^{-1}\left( 2\sin{\left(\frac{\pi}{4l-6} \right)}\cos{\left( \frac{\pi}{3i} \right)}\right)}{2\sin{\left(\frac{\pi}{4l-6} \right)} \cdot \sin^{-1}\left( \cos{\left( \frac{\pi}{3i} \right)} \right)} < \frac{3}{\pi} \frac{\sin^{-1}\left(2\sin{\left(\frac{\pi}{4l-6} \right)} \right)}{2\sin{\left(\frac{\pi}{4l-6} \right)}}.\\
     &\text{By simplifying then}\\
     &\sin^{-1}\left( 2\sin{\left(\frac{\pi}{4l-6} \right)}\cos{\left( \frac{\pi}{3i} \right)}\right) < \frac{3}{\pi} \sin^{-1}\left(2\sin{\left(\frac{\pi}{4l-6} \right)} \right) \sin^{-1}\left( \cos{\left( \frac{\pi}{3i} \right)} \right)\\
     &\text{since}~ \sin^{-1}\left( \cos{\left( \frac{\pi}{3i} \right)} \right)=\left( \frac{\pi}{2}-\frac{\pi}{3i} \right), \text{then}\\
     &\sin^{-1}\left( 2\sin{\left(\frac{\pi}{4l-6} \right)}\cos{\left( \frac{\pi}{3i} \right)}\right)
     <\frac{3}{\pi} \sin^{-1}\left(2\sin{\left(\frac{\pi}{4l-6} \right)} \right) \left( \frac{\pi}{2}-\frac{\pi}{3i} \right)
 \end{align*}
 as we claimed.
 \end{proof}
\section*{Acknowledgement}
I am very grateful to Prof. Emil Molnar for a bunch of meaningful mathematical discussions. I am also really thankful to Dr. Jen\H{o} Szirmai who guided my doctorate studies in Budapest. 



\begin{thebibliography}{99}
\bibitem{Delone}
{\sc Delone (Delaunay), B.N}, {\em Theory of plannigons (in Russian)}, Izv.Aked.Nauk SSSR Ser. Mat. {\bf 23}(3)(1959), 365--386.

\bibitem{luvcic1990combinatorial}
 {\sc Lu{\v{c}}i{\'c}, Z and Moln{\"a}r, E}, {\em Combinatorial classification of fundamental domains of finite area for planar discontinuous isometry groups}, Archiv der Mathematik {\bf 54}(1990),511--520.
 
 \bibitem{luvcic1991}
 {\sc Lu{\v{c}}i{\'c}, Z and Moln{\'a}r, E}, {\em Fundamental domains for planar discontinuous groups and uniform tilings},Geometriae Dedicata {\bf 40} (1991),125--143.
 
 \bibitem{luvcic2018}
 {\sc  Lu{\v{c}}i{\'c}, Z ,  Moln{\"a}r, E and Vasiljevic, N}, {\em An algorithm for classification of fundamental polygons for a plane discontinuous group}, Discrete Geometry and Symmetry, (2018), 257--278.
 
\bibitem{Macbeath}
{\sc Macbeath, A. M.}, {\em The classification of non-euclidean plane crystallographic groups}, Canadian Journal of Mathematics {\bf19} (1967), 1192--1205.

\bibitem{Magnus2019} 
 {\sc Magnus, Jan R and Neudecker, Heinz},{\em Matrix differential calculus with applications in statistics and econometrics},{John Wiley \& Sons},{2019}
 
 \bibitem{molnar2019} 
 {\sc Moln{\'a}r, E,  Prok, I and Szirmai, J}, {\em From a nice tiling to theory and applications}, Towards New Perspectives on Mathematics Education,{\bf 57} (2019), 85--106
 
 \bibitem{Trench2012}
 {\sc Trench, William F}, {\em The Method of Lagrange Multipliers}, {Research Gate, Book},{2012}
 \end{thebibliography}
\end{document}